\documentclass{amsart}
\usepackage{amsmath}
\usepackage{amssymb}
\usepackage{amsfonts}

\newtheorem{theorem}{Theorem}
\theoremstyle{plain}

\newtheorem{corollary}{Corollary}

\newtheorem{lemma}{Lemma}

\newtheorem{proposition}{Proposition}
\newtheorem{remark}{Remark}

\numberwithin{equation}{section}
\input{tcilatex}

\begin{document}
\title[Hermite-Hadamard type]{Some new integral inequalities for twice
differentiable convex mappings }
\author{Mehmet Zeki Sar\i kaya$^{\star }$}
\address{Department of Mathematics, Faculty of Science and Arts, D\"{u}zce
University, D\"{u}zce, Turkey}
\email{sarikayamz@gmail.com}
\thanks{$^{\star }$corresponding author}
\author{Huseyin Y\i ld\i r\i m}
\address{Department of Mathematics, Faculty of Science and Arts, Kahramanmara%
\c{s} S\"{u}t\c{c}\"{u} \.{I}mam University, Kahramanmara\c{s}, Turkey}
\email{hyildir@ksu.edu.tr}
\date{}
\subjclass[2000]{ 26D15, }
\keywords{Convex function, Hermite-Hadamard inequality, H\"{o}lder's
inequality.}

\begin{abstract}
In this paper, we establish several new inequalities for some twice
differantiable mappings that are connected with the celebrated
Hermite-Hadamard integral inequality. Some applications for special means of
real numbers are also provided.
\end{abstract}

\maketitle

\section{Introduction}

The following inequality is well known in the literature as the
Hermite-Hadamard integral inequality (see, \cite{PPT}):

\begin{equation}
f\left( \frac{a+b}{2}\right) \leq \frac{1}{b-a}\int_{a}^{b}f(x)dx\leq \frac{%
f(a)+f(b)}{2}  \label{H}
\end{equation}%
where $f:I\subset \mathbb{R}\rightarrow \mathbb{R}$ is a convex function on
the interval $I$ of real numbers and $a,b\in I$ with $a<b$. A function $%
f:[a,b]\subset \mathbb{R}\rightarrow \mathbb{R}$ is said to be convex if \
whenever $x,y\in \lbrack a,b]$ and $t\in \left[ 0,1\right] $, the following
inequality holds%
\begin{equation*}
f(tx+(1-t)y)\leq tf(x)+(1-t)f(y).
\end{equation*}%
This definition has its origins in Jensen's results from \cite{jensen} and
has opened up the most extended, useful and multi-disciplinary domain of
mathematics, namely, canvex analysis. Convex curvers and convex bodies have
appeared in mathematical literature since antiquity and there are many
important resuls related to them. We say that $f$ is concave if $(-f)$ is
convex.

A largely applied inequality for convex functions, due to its geometrical
significance, is Hadamard's inequality, (see \cite{SSDRPA},\cite{Hussain}, 
\cite{USK} -\cite{sarikaya}) which has generated a wide range of directions
for extension and a rich mathematical literature.

In \cite{SSDRPA}, Dragomir and Agarwal established the following result
connected with the right part of (\ref{H}) as well as to apply them for some
elementary inequalities for real numbers and numerical integration:

\begin{theorem}
Let $f:I^{\circ }\subset \mathbb{R}\rightarrow \mathbb{R}$ be a
differentiable mapping on $I^{\circ }$, $a,b\in I^{\circ }$ with $a<b,$and $%
f^{\prime }\in L(a,b).$ If the mapping $\left\vert f^{\prime }\right\vert $
is convex on $\left[ a,b\right] $, then the following inequality holds:%
\begin{equation}
\left\vert \dfrac{f(a)+f(b)}{2}-\dfrac{1}{b-a}\dint_{a}^{b}f(x)dx\right\vert
\leq \left( b-a\right) \left( \frac{\left\vert f^{\prime }(a)\right\vert
+\left\vert f^{\prime }(b)\right\vert }{8}\right) .  \label{1H}
\end{equation}
\end{theorem}

In \cite{CEMPJP}, Pearce and Pe\v{c}ari\'{c} proved the following theorem:

\begin{theorem}
Let $f:I\subset \mathbb{R}\rightarrow \mathbb{R}$, be a differentiable
mapping on $I^{\circ }$, $a,b\in I^{\circ }$ with $a<b$. If the mapping $%
\left\vert f^{\prime }\right\vert ^{q}$ is convex on $\left[ a,b\right] $
for some $q\geq 1$, then%
\begin{equation}
\left\vert \dfrac{f(a)+f(b)}{2}-\dfrac{1}{b-a}\dint_{a}^{b}f(x)dx\right\vert
\leq \frac{b-a}{4}\left( \frac{\left\vert f^{\prime }(a)\right\vert
^{q}+\left\vert f^{\prime }(b)\right\vert ^{q}}{2}\right) ^{\frac{1}{q}}
\label{2H}
\end{equation}%
and%
\begin{equation}
\left\vert \frac{1}{b-a}\int_{a}^{b}f(x)dx-f\left( \frac{a+b}{2}\right)
\right\vert \leq \frac{b-a}{4}\left( \frac{\left\vert f^{\prime
}(a)\right\vert ^{q}+\left\vert f^{\prime }(b)\right\vert ^{q}}{2}\right) ^{%
\frac{1}{q}}.  \label{3H}
\end{equation}
\end{theorem}

Also, in \cite{USK}, K\i rmac\i\ obtained the following inequality for
differeftiable mappings which are connected with Hermite-Hadamard's
inequality:

\begin{theorem}
\label{l2} Let $f:I^{\circ }\subset \mathbb{R}\rightarrow \mathbb{R}$ be a
differentiable mapping on $I^{\circ }$, $a,b\in I^{\circ }$ with $a<b$. If
the mapping $\left\vert f^{\prime }\right\vert $ is convex on $\left[ a,b%
\right] $, then we have 
\begin{equation}
\left\vert \frac{1}{b-a}\int_{a}^{b}f(x)dx-f\left( \frac{a+b}{2}\right)
\right\vert \leq \frac{b-a}{8}\left( \left\vert f^{\prime }(a)\right\vert
+\left\vert f^{\prime }(b)\right\vert \right) .  \label{4H}
\end{equation}
\end{theorem}

In \cite{sarikaya}, Sarikaya et. al. established inequalities for twice
differentiable convex mappings which are connected with Hadamard's
inequality, and they used the following lemma to prove their results:

\begin{lemma}
Let $f:I^{\circ }\subset \mathbb{R}\rightarrow \mathbb{R}$ be twice
differentiable function on $I^{\circ }$, $a,b\in I^{\circ }$ with $a<b.$ If $%
f^{\prime \prime }\in L_{1}[a,b]$, then%
\begin{equation*}
\begin{array}{l}
\dfrac{1}{b-a}\dint_{a}^{b}f(x)dx-f(\dfrac{a+b}{2}) \\ 
\\ 
\ \ \ \ \ \ \ \ \ \ =\dfrac{\left( b-a\right) ^{2}}{2}\dint_{0}^{1}m\left(
t\right) \left[ f^{\prime \prime }(ta+(1-t)b)+f^{\prime \prime }(tb+(1-t)a)%
\right] dt,%
\end{array}%
\end{equation*}%
where%
\begin{equation*}
m(t):=\left\{ 
\begin{array}{ll}
t^{2} & ,t\in \lbrack 0,\frac{1}{2}) \\ 
&  \\ 
\left( 1-t\right) ^{2} & ,t\in \lbrack \frac{1}{2},1].%
\end{array}%
\right.
\end{equation*}
\end{lemma}

Also, the main inequalities in \cite{sarikaya}, pointed out as follows:

\begin{theorem}
Let $f:I\subset \mathbb{R}\rightarrow \mathbb{R}$ be twice differentiable
function on $I^{\circ }$ with $f^{\prime \prime }\in L_{1}[a,b]$. If $%
\left\vert f^{\prime \prime }\right\vert $ is convex on $[a,b],$\ then%
\begin{equation}
\begin{array}{l}
\left\vert \dfrac{1}{b-a}\dint_{a}^{b}f(x)dx-f(\dfrac{a+b}{2})\right\vert
\leq \dfrac{\left( b-a\right) ^{2}}{24}\left[ \dfrac{\left\vert f^{\prime
\prime }\left( a\right) \right\vert +\left\vert f^{\prime \prime }\left(
b\right) \right\vert }{2}\right] .%
\end{array}
\label{H1}
\end{equation}
\end{theorem}

\begin{theorem}
Let $f:I\subset \mathbb{R}\rightarrow \mathbb{R}$ be twice differentiable
function on $I^{\circ }$ such that $f^{\prime \prime }\in L_{1}[a,b]$ where $%
a,b\in I,$ $a<b$. If $\left\vert f^{\prime \prime }\right\vert ^{q}$ is
convex on $[a,b],$\ $q>1$, then%
\begin{equation}
\begin{array}{l}
\left\vert \dfrac{1}{b-a}\dint_{a}^{b}f(x)dx-f(\dfrac{a+b}{2})\right\vert
\leq \dfrac{\left( b-a\right) ^{2}}{8\left( 2p+1\right) ^{1/p}}\left[ \dfrac{%
\left\vert f^{\prime \prime }(a)\right\vert ^{q}+\left\vert f^{\prime \prime
}(b)\right\vert ^{q}}{2}\right] ^{1/q}%
\end{array}
\label{H2}
\end{equation}%
where $\frac{1}{p}+\frac{1}{q}=1.$
\end{theorem}

In \cite{Hussain}, Hussain et. al. proved some inequalities related to
Hermite-Hadamard's inequality for $s$-convex functions:

\begin{theorem}
Let $f:I\subset \lbrack 0,\infty )\rightarrow \mathbb{R}$ be twice
differentiable mapping on $I^{\circ }$ such that $f^{\prime \prime }\in L_{1}%
\left[ a,b\right] $ where $a,b\in I$ with $a<b.$ If $\left\vert f^{\prime
\prime }\right\vert $ is $s-$convex on $\left[ a,b\right] $ for some fixed $%
s\in \lbrack 0,1]$ and $q\geq 1,$ then the following inequality holds:%
\begin{equation}
\begin{array}{l}
\left\vert \dfrac{f(a)+f(b)}{2}-\dfrac{1}{b-a}\dint_{a}^{b}f(x)dx\right\vert
\leq \dfrac{\left( b-a\right) ^{2}}{2\times 6^{\frac{1}{p}}}\left[ \dfrac{%
\left\vert f^{\prime \prime }(a)\right\vert ^{q}+\left\vert f^{\prime \prime
}(b)\right\vert ^{q}}{(s+2)(s+3)}\right] ^{\frac{1}{q}}%
\end{array}
\label{H3}
\end{equation}%
where $\frac{1}{p}+\frac{1}{q}=1.$
\end{theorem}

\begin{remark}
If we take $s=1$ in (\ref{H3}), then we have%
\begin{equation}
\begin{array}{l}
\left\vert \dfrac{f(a)+f(b)}{2}-\dfrac{1}{b-a}\dint_{a}^{b}f(x)dx\right\vert
\leq \dfrac{\left( b-a\right) ^{2}}{12}\left[ \dfrac{\left\vert f^{\prime
\prime }(a)\right\vert ^{q}+\left\vert f^{\prime \prime }(b)\right\vert ^{q}%
}{2}\right] ^{\frac{1}{q}}.%
\end{array}
\label{H4}
\end{equation}
\end{remark}

In \cite{K}, Kirmanci proved the generalization identity connected with
Hermite-Hadamard inegral inequality for differentiable convex functions and
established the following results:

\begin{theorem}
\label{zz} Let $f:I^{\circ }\subset \mathbb{R}\rightarrow \mathbb{R}$ be a
differentiable mapping on $I^{\circ }$ $a,A,c,B,b$ $\in I^{\circ }$with $%
a\leq A\leq c\leq B\leq b$ and $p>1$. If the mapping $\left\vert f^{\prime
}\right\vert ^{p/(p-1)}$ is convex on $[a,b]$, then we have

\textbf{i.}%
\begin{equation*}
f(ca+(1-c)b)(B-A)+f(a)(1-B)+f(b)A-\dfrac{1}{b-a}\dint_{a}^{b}f(x)dx=(a-b)%
\dint_{0}^{1}S(t)f(ta+(1-t)b)dx,
\end{equation*}

\textbf{ii.}%
\begin{eqnarray}
&&\left\vert \frac{1}{a-b}\left[ f(ca+(1-c)b)(B-A)+f(a)(1-B)+f(b)A-\dfrac{1}{%
b-a}\dint_{a}^{b}f(x)dx\right] \right\vert   \label{H5} \\
&\leq &\left[ \frac{A^{p+1}+(c-A)^{p+1}}{p+1}\right] ^{1/p}\left( \dfrac{%
c^{2}\left\vert f^{\prime }(a)\right\vert ^{q}+(2c-c^{2})\left\vert
f^{\prime }(b)\right\vert ^{q}}{2}\right)   \notag \\
&&+\left[ \frac{(B-c)^{p+1}+(1-B)^{p+1}}{p+1}\right] ^{1/p}\left( \dfrac{%
(1-c^{2})\left\vert f^{\prime }(a)\right\vert ^{q}+(1-c)^{2}\left\vert
f^{\prime }(b)\right\vert ^{q}}{2}\right) ,  \notag
\end{eqnarray}%
where 
\begin{equation*}
S(t)=\left\{ 
\begin{array}{c}
t-A,\ \ t\in \lbrack 0,c] \\ 
t-B\ \ t\in (c,1].%
\end{array}%
\right. 
\end{equation*}

\begin{corollary}
Under the assumptations of Theorem $\ref{zz}$ with $A=B=c=\frac{1}{2},$ we
have 
\begin{equation}
\begin{array}{l}
\left\vert \dfrac{1}{b-a}\dint_{a}^{b}f(x)dx-\dfrac{f\left( a\right)
+f\left( b\right) }{2}\right\vert  \\ 
\\ 
\text{ \ \ \ \ \ \ \ \ \ }\leq \dfrac{\left( b-a\right) }{4\left( p+1\right)
^{1/p}}\left\{ \left( \dfrac{\left\vert f^{\prime \prime }\left( a\right)
\right\vert ^{q}+3\left\vert f^{\prime \prime }\left( b\right) \right\vert
^{q}}{4}\right) ^{1/q}+\left( \dfrac{3\left\vert f^{\prime \prime }\left(
a\right) \right\vert ^{q}+\left\vert f^{\prime \prime }\left( b\right)
\right\vert ^{q}}{4}\right) ^{1/q}\right\} 
\end{array}
\label{H6}
\end{equation}%
and if we take$\ A=0,\ B=1,\ c=\frac{1}{2}$ in Theorem $\ref{zz},$ then it
follows that%
\begin{equation}
\begin{array}{l}
\left\vert \dfrac{1}{b-a}\dint_{a}^{b}f(x)dx-f\left( \dfrac{a+b}{2}\right)
\right\vert  \\ 
\\ 
\text{\ \ \ \ \ \ \ }\leq \dfrac{\left( b-a\right) }{4\left( p+1\right)
^{1/p}}\left\{ \left( \dfrac{\left\vert f^{\prime \prime }\left( a\right)
\right\vert ^{q}+3\left\vert f^{\prime \prime }\left( b\right) \right\vert
^{q}}{4}\right) ^{1/q}+\left( \dfrac{3\left\vert f^{\prime \prime }\left(
a\right) \right\vert ^{q}+\left\vert f^{\prime \prime }\left( b\right)
\right\vert ^{q}}{4}\right) ^{1/q}\right\} .%
\end{array}
\label{H7}
\end{equation}
\end{corollary}
\end{theorem}

In this article, using functions whose twice derivatives absolute values are
convex, we obtained new inequalities releted to the left and right hand
sides of Hermite-Hadamard inequality. Finally, we gave some applications for
special means of real numbers.

\section{Main Results}

Throughout, we suppose $I$ is an interval on $\mathbb{R}$ and $a,b,c,d,y\in
I $ with $a\leq c\leq y\leq d\leq b,\ (y\neq a,b).$ We start with the
following lemma:

\begin{lemma}
\label{lm} Let $f:I\subset \mathbb{R}\rightarrow \mathbb{R}$ be twice
differentiable function on $I^{\circ }$ with $f^{\prime \prime }\in
L_{1}[a,b]$, then%
\begin{equation}
\begin{array}{l}
\dfrac{\left( d-y\right) ^{2}-\left( c-y\right) ^{2}}{2}\left( a-b\right)
f^{\prime }(ya+(1-y)b)+\dfrac{c^{2}f^{\prime }(b)-\left( d-1\right)
^{2}f^{\prime }(a)}{2}\left( a-b\right)  \\ 
\\ 
\text{ \ \ \ \ \ \ \ \ \ }+\left( d-c\right) f(ya+(1-y)b)+\left(
cf(a)-\left( d-1\right) f(b)\right) +\dfrac{1}{b-a}\dint_{a}^{b}f(x)dx \\ 
\\ 
\text{ \ \ \ \ \ \ \ \ \ }=\left( b-a\right) ^{2}\dint_{0}^{1}k\left(
t\right) f^{\prime \prime }(ta+(1-t)b)dt,%
\end{array}
\label{z}
\end{equation}%
where%
\begin{equation*}
k(t):=\left\{ 
\begin{array}{ll}
\dfrac{\left( c-t\right) ^{2}}{2} & ,t\in \lbrack 0,y) \\ 
&  \\ 
\dfrac{\left( d-t\right) ^{2}}{2} & ,t\in \lbrack y,1].%
\end{array}%
\right. 
\end{equation*}
\end{lemma}

\begin{proof}
It suffices to note that%
\begin{equation*}
\begin{array}{lll}
I & = & \dint_{0}^{1}k\left( t\right) f^{\prime \prime }(ta+(1-t)b)dt \\ 
&  &  \\ 
& = & \dint_{0}^{y}\dfrac{\left( c-t\right) ^{2}}{2}f^{\prime \prime
}(ta+(1-t)b)dt+\dint_{y}^{1}\dfrac{\left( d-t\right) ^{2}}{2}f^{\prime
\prime }(ta+(1-t)b)dt \\ 
&  &  \\ 
& = & I_{1}+I_{2}.%
\end{array}%
\end{equation*}%
By inegration by parts, we have the following identity%
\begin{equation*}
\begin{array}{lll}
I_{1} & = & \dint_{0}^{y}\dfrac{\left( c-t\right) ^{2}}{2}f^{\prime \prime
}(ta+(1-t)b)dt \\ 
&  &  \\ 
& = & \dfrac{\left( c-t\right) ^{2}}{2\left( a-b\right) }f^{\prime
}(ta+(1-t)b)\underset{0}{\overset{y}{\mid }}+\dfrac{1}{a-b}%
\dint_{0}^{y}\left( c-t\right) f^{\prime }(ta+(1-t)b)dt \\ 
&  &  \\ 
& = & \dfrac{\left( c-y\right) ^{2}}{2\left( a-b\right) }f^{\prime
}(ya+(1-y)b)-\dfrac{c^{2}}{2\left( a-b\right) }f^{\prime }(b)+\dfrac{1}{a-b}%
\dint_{0}^{y}\left( c-t\right) f^{\prime }(ta+(1-t)b)dt \\ 
&  &  \\ 
& = & \dfrac{\left( c-y\right) ^{2}}{2\left( a-b\right) }f^{\prime
}(ya+(1-y)b)-\dfrac{c^{2}}{2\left( a-b\right) }f^{\prime }(b) \\ 
&  &  \\ 
& + & \dfrac{1}{a-b}\left[ \dfrac{c-t}{a-b}f(ta+(1-t)b)\underset{0}{\overset{%
y}{\mid }}+\dfrac{1}{a-b}\dint_{0}^{y}f(ta+(1-t)b)dt\right]  \\ 
&  &  \\ 
& = & \dfrac{\left( c-y\right) ^{2}}{2\left( a-b\right) }f^{\prime
}(ya+(1-y)b)-\dfrac{c^{2}}{2\left( a-b\right) }f^{\prime }(b) \\ 
&  &  \\ 
& + & \dfrac{c-y}{\left( b-a\right) ^{2}}f(ya+(1-y)b)-\dfrac{c}{\left(
b-a\right) ^{2}}f(b)+\dfrac{1}{\left( b-a\right) ^{2}}%
\dint_{0}^{y}f(ta+(1-t)b)dt.%
\end{array}%
\end{equation*}%
Similarly, we observe that%
\begin{equation*}
\begin{array}{lll}
I_{2} & = & \dfrac{\left( d-1\right) ^{2}}{2\left( a-b\right) }f^{\prime
}(a)-\dfrac{\left( d-y\right) ^{2}}{2\left( a-b\right) }f^{\prime
}(ya+(1-y)b) \\ 
&  &  \\ 
& + & \dfrac{d-1}{\left( b-a\right) ^{2}}f(a)-\dfrac{d-y}{\left( b-a\right)
^{2}}f(ya+(1-y)b)+\dfrac{1}{\left( b-a\right) ^{2}}%
\dint_{y}^{1}f(ta+(1-t)b)dt.%
\end{array}%
\end{equation*}%
Thus, we can write%
\begin{equation*}
\begin{array}{lll}
I & = & I_{1}+I_{2} \\ 
&  &  \\ 
& = & \dfrac{\left( c-y\right) ^{2}-\left( d-y\right) ^{2}}{2\left(
a-b\right) }f^{\prime }(ya+(1-y)b)+\dfrac{\left( d-1\right) ^{2}f^{\prime
}(a)-c^{2}f^{\prime }(b)}{2\left( a-b\right) } \\ 
&  &  \\ 
& + & \dfrac{c-d}{\left( b-a\right) ^{2}}f(ya+(1-y)b)+\dfrac{\left(
d-1\right) f(a)-cf(b)}{\left( b-a\right) ^{2}}+\dfrac{1}{\left( b-a\right)
^{2}}\dint_{0}^{1}f(ta+(1-t)b)dt.%
\end{array}%
\end{equation*}%
Using the change of the variable $x=ta+(1-t)b$ for $t\in \left[ 0,1\right] $
and by multiplying the both sides by $\left( b-a\right) ^{2}$ which give the
required identity (\ref{z}).
\end{proof}

Now, by using the above lemma, we prove our main theorems:

\begin{theorem}
\label{thm1} Let $f:I\subset \mathbb{R}\rightarrow \mathbb{R}$ be twice
differentiable function on $I^{\circ }$ such that $f^{\prime \prime }\in
L_{1}[a,b]$ where $a,b\in I,$ $a<b$. If $\left\vert f^{\prime \prime
}\right\vert $ is convex on $[a,b]$ then the following inequality holds:%
\begin{equation*}
\begin{array}{l}
\left\vert \dfrac{\left( c-y\right) ^{2}-\left( d-y\right) ^{2}}{2}\left(
a-b\right) f^{\prime }(ya+(1-y)b)+\dfrac{\left( d-1\right) ^{2}f^{\prime
}(a)-c^{2}f^{\prime }(b)}{2}\left( a-b\right) \right.  \\ 
\\ 
\text{ \ \ \ \ \ \ \ \ \ }\left. +\left( c-d\right) f(ya+(1-y)b)+\left[
\left( d-1\right) f(a)-cf(b)\right] +\dfrac{1}{b-a}\dint_{a}^{b}f(x)dx\right%
\vert  \\ 
\\ 
\text{ \ \ \ \ \ \ \ \ \ }\leq \dfrac{\left( b-a\right) ^{2}}{24}\left(
A\left\vert f^{\prime \prime }(a)\right\vert +B\left\vert f^{\prime \prime
}(b)\right\vert \right) ,%
\end{array}%
\end{equation*}%
where 
\begin{equation*}
A=6d^{2}-8d+3+\left( 6c^{2}-6d^{2}\right) y^{2}+\left( 8d-8c\right) y^{3}
\end{equation*}%
and 
\begin{equation*}
B=6d^{2}-4d+1+\left( 12c^{2}-12d^{2}\right) y+\left(
12d+6d^{2}-12c-6c^{2}\right) y^{2}+\left( 8c-8d\right) y^{3}.
\end{equation*}
\end{theorem}

\begin{proof}
From Lemma $\ref{lm}$ and the definition of $k(t),$ we get, 
\begin{equation*}
\begin{array}{l}
\left\vert \dfrac{\left( c-y\right) ^{2}-\left( d-y\right) ^{2}}{2}\left(
a-b\right) f^{\prime }(ya+(1-y)b)+\dfrac{\left( d-1\right) ^{2}f^{\prime
}(a)-c^{2}f^{\prime }(b)}{2}\left( a-b\right) \right.  \\ 
\\ 
\text{ \ \ \ \ \ \ \ \ \ }\left. +\left( c-d\right) f(ya+(1-y)b)+\left[
\left( d-1\right) f(a)-cf(b)\right] +\dfrac{1}{b-a}\dint_{a}^{b}f(x)dx\right%
\vert  \\ 
\\ 
\text{ \ \ \ \ \ \ \ \ \ }\leq \dfrac{\left( b-a\right) ^{2}}{2}\left\{
\dint_{0}^{y}\left( c-t\right) ^{2}\left\vert f^{\prime \prime
}(ta+(1-t)b)\right\vert dt+\dint_{y}^{1}\left( d-t\right) ^{2}\left\vert
f^{\prime \prime }(ta+(1-t)b)\right\vert dt\right\} .%
\end{array}%
\end{equation*}%
By the convexity of $\left\vert f^{\prime \prime }\right\vert $, we get%
\begin{equation*}
\begin{array}{l}
\left\vert \dfrac{\left( c-y\right) ^{2}-\left( d-y\right) ^{2}}{2}\left(
a-b\right) f^{\prime }(ya+(1-y)b)+\dfrac{\left( d-1\right) ^{2}f^{\prime
}(a)-c^{2}f^{\prime }(b)}{2}\left( a-b\right) \right.  \\ 
\\ 
\text{ \ \ \ \ \ \ \ \ \ }\left. +\left( c-d\right) f(ya+(1-y)b)+\left[
\left( d-1\right) f(a)-cf(b)\right] +\dfrac{1}{b-a}\dint_{a}^{b}f(x)dx\right%
\vert  \\ 
\\ 
\text{ \ \ \ \ \ \ \ \ \ }\leq \dfrac{\left( b-a\right) ^{2}}{2}\left\{
\dint_{0}^{y}\left( c-t\right) ^{2}t\left\vert f^{\prime \prime
}(a)\right\vert dt+\dint_{0}^{y}\left( c-t\right) ^{2}\left( 1-t\right)
\left\vert f^{\prime \prime }(b)\right\vert dt\right.  \\ 
\\ 
\text{ \ \ \ \ \ \ \ \ \ }+\left. \dint_{y}^{1}\left( d-t\right)
^{2}t\left\vert f^{\prime \prime }(a)\right\vert dt+\dint_{y}^{1}\left(
d-t\right) ^{2}\left( 1-t\right) \left\vert f^{\prime \prime }(b)\right\vert
dt\right\} 
\end{array}%
\end{equation*}%
\begin{equation*}
\begin{array}{l}
\text{ \ \ \ \ \ \ \ \ \ }=\dfrac{\left( b-a\right) ^{2}}{2}\left\{ \dfrac{%
6c^{2}y^{2}-8cy^{3}+3y^{4}}{12}\left\vert f^{\prime \prime }(a)\right\vert
\right.  \\ 
\\ 
\text{ \ \ \ \ \ \ \ \ \ }+\dfrac{12c^{2}y+\left( -12c-6c^{2}\right)
y^{2}+\left( 4+8c\right) y^{3}-3y^{4}}{12}\left\vert f^{\prime \prime
}(b)\right\vert  \\ 
\\ 
\text{ \ \ \ \ \ \ \ \ \ }+\dfrac{6d^{2}-8d+3-6d^{2}y^{2}+8dy^{3}-3y^{4}}{12}%
\left\vert f^{\prime \prime }(a)\right\vert + \\ 
\\ 
\text{ \ \ \ \ \ \ \ \ \ }+\left. \dfrac{6d^{2}-4d+1-12d^{2}y+\left(
12d+6d^{2}\right) y^{2}+\left( 4-8d\right) y^{3}+3y^{4}}{12}\left\vert
f^{\prime \prime }(b)\right\vert \right\}  \\ 
\\ 
\text{ \ \ \ \ \ \ \ \ \ }=\dfrac{\left( b-a\right) ^{2}}{2}\left\{ \dfrac{%
6d^{2}-8d+3+\left( 6c^{2}-6d^{2}\right) y^{2}+\left( 8d-8c\right) y^{3}}{12}%
\left\vert f^{\prime \prime }(a)\right\vert \right.  \\ 
\\ 
\text{ \ \ \ \ \ \ \ \ \ }+\left. \dfrac{6d^{2}-4d+1+\left(
12c^{2}-12d^{2}\right) y+\left( 12d+6d^{2}-12c-6c^{2}\right) y^{2}+\left(
8c-8d\right) y^{3}}{12}\left\vert f^{\prime \prime }(b)\right\vert \right\} ,%
\end{array}%
\end{equation*}%
which completes the proof.
\end{proof}

\begin{corollary}
\label{c1} Under the assumptions of Theorem $\ref{thm1}$ with $y=\dfrac{1}{2}%
,$ $c=0,$ $d=1,$ we have%
\begin{equation}
\left\vert \dfrac{1}{b-a}\dint_{a}^{b}f(x)dx-f(\dfrac{a+b}{2})\right\vert
\leq \dfrac{(b-a)^{2}}{48}(\left\vert f^{\prime \prime }\left( a\right)
\right\vert +\left\vert f^{\prime \prime }\left( b\right) \right\vert )
\label{H8}
\end{equation}%
and if we take $y=c=d=\dfrac{1}{2}$ and $f^{\prime }(a)=f^{\prime }(b)$ in
Theorem $\ref{thm1},$ we have 
\begin{equation}
\begin{array}{l}
\left\vert \dfrac{1}{b-a}\dint_{a}^{b}f(x)dx-\dfrac{f\left( a\right)
+f\left( b\right) }{2}\right\vert \leq \dfrac{(b-a)^{2}}{48}(\left\vert
f^{\prime \prime }\left( a\right) \right\vert +\left\vert f^{\prime \prime
}\left( b\right) \right\vert ).%
\end{array}
\label{H9}
\end{equation}
\end{corollary}

\begin{remark}
We note that the obtained midpoint inequalities (\ref{H8}) and (\ref{H9})
are better than the inequalities (\ref{H1}) and (\ref{H4}), respectively.
\end{remark}

\begin{remark}
We note that the obtained midpoint inequality (\ref{H8}) is the same
mitpoint in inequality (\ref{H1}).
\end{remark}

Another similar result may be extended in the following theorem

\begin{theorem}
\label{thm2} Let $f:I\subset \mathbb{R}\rightarrow \mathbb{R}$ be twice
differentiable function on $I^{\circ }$ such that $f^{\prime \prime }\in
L_{1}[a,b]$ where $a,b\in I,$ $a<b$. If $\left\vert f^{\prime \prime
}\right\vert ^{q}$ is convex on $[a,b],$\ $q>1$, then%
\begin{equation*}
\begin{array}{l}
\left\vert \dfrac{\left( c-y\right) ^{2}-\left( d-y\right) ^{2}}{2}\left(
a-b\right) f^{\prime }(ya+(1-y)b)+\dfrac{\left( d-1\right) ^{2}f^{\prime
}(a)-c^{2}f^{\prime }(b)}{2}\left( a-b\right) \right. \\ 
\\ 
\text{ \ \ \ \ \ \ \ \ \ }\left. +\left( c-d\right) f(ya+(1-y)b)+\left[
\left( d-1\right) f(a)-cf(b)\right] +\dfrac{1}{b-a}\dint_{a}^{b}f(x)dx\right%
\vert \\ 
\\ 
\text{ \ \ \ \ \ \ \ \ \ }\leq \dfrac{\left( b-a\right) ^{2}}{2}\left( 
\dfrac{1}{2p+1}\right) ^{1/p}\left\{ \left( c^{2p+1}+\left( y-c\right)
^{2p+1}\right) ^{1/p}\left( \dfrac{y^{2}\left\vert f^{\prime \prime }\left(
a\right) \right\vert ^{q}+\left( 2y-y^{2}\right) \left\vert f^{\prime \prime
}\left( b\right) \right\vert ^{q}}{2}\right) ^{1/q}\right. \\ 
\\ 
\text{ \ \ \ \ \ \ \ \ \ }+\left. \left( \left( d-y\right) ^{2p+1}+\left(
1-d\right) ^{2p+1}\right) ^{1/p}\left( \dfrac{\left( 1-y^{2}\right)
\left\vert f^{\prime \prime }\left( a\right) \right\vert ^{q}+\left(
1-y\right) ^{2}\left\vert f^{\prime \prime }\left( b\right) \right\vert ^{q}%
}{2}\right) ^{1/q}\right\} .%
\end{array}%
\end{equation*}
\end{theorem}

\begin{proof}
From Lemma $\ref{lm}$, by the definition $k\left( t\right) $ and using by H%
\"{o}lder's inequality, it follows that%
\begin{equation*}
\begin{array}{l}
\left\vert \dfrac{\left( c-y\right) ^{2}-\left( d-y\right) ^{2}}{2}\left(
a-b\right) f^{\prime }(ya+(1-y)b)+\dfrac{\left( d-1\right) ^{2}f^{\prime
}(a)-c^{2}f^{\prime }(b)}{2}\left( a-b\right) \right. \\ 
\\ 
\text{ \ \ \ \ \ \ \ \ \ }\left. +\left( c-d\right) f(ya+(1-y)b)+\left[
\left( d-1\right) f(a)-cf(b)\right] +\dfrac{1}{b-a}\dint_{a}^{b}f(x)dx\right%
\vert \\ 
\\ 
\text{ \ \ \ \ \ \ \ \ \ }\leq \dfrac{\left( b-a\right) ^{2}}{2}\left\{
\dint_{0}^{y}\left( c-t\right) ^{2}\left\vert f^{\prime \prime
}(ta+(1-t)b)dt\right\vert dt+\dint_{y}^{1}\left( d-t\right) ^{2}\left\vert
f^{\prime \prime }(ta+(1-t)b)dt\right\vert dt\right\} \\ 
\\ 
\text{ \ \ \ \ \ \ \ \ \ }\leq \dfrac{\left( b-a\right) ^{2}}{2}\left\{
\left( \dint_{0}^{y}\left\vert c-t\right\vert ^{2p}dt\right) ^{1/p}\left(
\dint_{0}^{y}\left\vert f^{\prime \prime }(ta+(1-t)b)\right\vert
^{q}dt\right) ^{1/q}\right. \\ 
\\ 
\text{ \ \ \ \ \ \ \ \ \ }+\left. \left( \dint_{y}^{1}\left\vert
d-t\right\vert ^{2p}dt\right) ^{1/p}\left( \dint_{y}^{1}\left\vert f^{\prime
\prime }(ta+(1-t)b)\right\vert ^{q}dt\right) ^{1/q}\right\} . \\ 
\end{array}%
\end{equation*}%
Since $\left\vert f^{\prime \prime }\right\vert ^{q}$ is convex on $\left(
a,b\right) ,$ we known that for $t\in \left[ 0,1\right] ,$ 
\begin{equation*}
\left\vert f^{\prime \prime }(ta+(1-t)b)\right\vert ^{q}\leq t\left\vert
f^{\prime \prime }(a)\right\vert ^{q}+\left( 1-t\right) \left\vert f^{\prime
\prime }(b)\right\vert ^{q}
\end{equation*}%
\begin{equation*}
\begin{array}{l}
\left\vert \dfrac{\left( c-y\right) ^{2}-\left( d-y\right) ^{2}}{2}\left(
a-b\right) f^{\prime }(ya+(1-y)b)+\dfrac{\left( d-1\right) ^{2}f^{\prime
}(a)-c^{2}f^{\prime }(b)}{2}\left( a-b\right) \right. \\ 
\\ 
\text{ \ \ \ \ \ \ \ \ \ }\left. +\left( c-d\right) f(ya+(1-y)b)+\left[
\left( d-1\right) f(a)-cf(b)\right] +\dfrac{1}{b-a}\dint_{a}^{b}f(x)dx\right%
\vert \\ 
\\ 
\text{ \ \ \ \ \ \ \ \ \ }\leq \dfrac{\left( b-a\right) ^{2}}{2}\left\{
\left( \dint_{0}^{y}\left\vert c-t\right\vert ^{2p}dt\right) ^{1/p}\left(
\dint_{0}^{y}\left\vert f^{\prime \prime }(ta+(1-t)b)\right\vert
^{q}dt\right) ^{1/q}\right. \\ 
\\ 
\text{ \ \ \ \ \ \ \ \ \ }+\left. \left( \dint_{y}^{1}\left\vert
d-t\right\vert ^{2p}dt\right) ^{1/p}\left( \dint_{y}^{1}\left\vert f^{\prime
\prime }(ta+(1-t)b)\right\vert ^{q}dt\right) ^{1/q}\right\} \\ 
\\ 
\text{ \ \ \ \ \ \ \ \ \ }\leq \dfrac{\left( b-a\right) ^{2}}{2}\left\{
\left( \dint_{0}^{y}\left\vert c-t\right\vert ^{2p}dt\right) ^{1/p}\left(
\dint_{0}^{y}\left( t\left\vert f^{\prime \prime }(a)\right\vert ^{q}+\left(
1-t\right) \left\vert f^{\prime \prime }(b)\right\vert ^{q}\right) dt\right)
^{1/q}\right. \\ 
\\ 
\text{ \ \ \ \ \ \ \ \ \ }+\left. \left( \dint_{y}^{1}\left\vert
d-t\right\vert ^{2p}dt\right) ^{1/p}\left( \dint_{y}^{1}\left( t\left\vert
f^{\prime \prime }(a)\right\vert ^{q}+\left( 1-t\right) \left\vert f^{\prime
\prime }(b)\right\vert ^{q}\right) dt\right) ^{1/q}\right\} \\ 
\\ 
\text{ \ \ \ \ \ \ \ \ \ }=\dfrac{\left( b-a\right) ^{2}}{2}\left( \dfrac{1}{%
2p+1}\right) ^{1/p}\left\{ \left( c^{2p+1}+\left( y-c\right) ^{2p+1}\right)
^{1/p}\left( \dfrac{y^{2}\left\vert f^{\prime \prime }\left( a\right)
\right\vert ^{q}+\left( 2y-y^{2}\right) \left\vert f^{\prime \prime }\left(
b\right) \right\vert ^{q}}{2}\right) ^{1/q}\right. \\ 
\\ 
\text{ \ \ \ \ \ \ \ \ \ }+\left. \left( \left( d-y\right) ^{2p+1}+\left(
1-d\right) ^{2p+1}\right) ^{1/p}\left( \dfrac{\left( 1-y^{2}\right)
\left\vert f^{\prime \prime }\left( a\right) \right\vert ^{q}+\left(
1-y\right) ^{2}\left\vert f^{\prime \prime }\left( b\right) \right\vert ^{q}%
}{2}\right) ^{1/q}\right\} ,%
\end{array}%
\end{equation*}%
where we have used the facts that%
\begin{equation*}
\begin{array}{lll}
\dint_{0}^{y}\left\vert c-t\right\vert ^{2p}dt & =\dint_{0}^{c}(c-t)^{2p}dt+%
\dint_{c}^{y}(t-c)^{2p}dt & =\dfrac{1}{2p+1}\left( c^{2p+1}+\left(
y-c\right) ^{2p+1}\right) \\ 
&  &  \\ 
\dint_{y}^{1}\left\vert d-t\right\vert ^{2p}dt & =\dint_{y}^{d}(d-t)^{2p}dt+%
\dint_{d}^{1}(t-d)^{2p}dt & =\dfrac{1}{2p+1}\left( \left( d-y\right)
^{2p+1}+\left( 1-d\right) ^{2p+1}\right) ,%
\end{array}%
\end{equation*}%
and%
\begin{equation*}
\begin{array}{lll}
\dint_{0}^{y}\left( t\left\vert f^{\prime \prime }(a)\right\vert ^{q}+\left(
1-t\right) \left\vert f^{\prime \prime }(b)\right\vert ^{q}\right) dt & = & 
\dfrac{y^{2}\left\vert f^{\prime \prime }\left( a\right) \right\vert
^{q}+\left( 2y-y^{2}\right) \left\vert f^{\prime \prime }\left( b\right)
\right\vert ^{q}}{2} \\ 
&  &  \\ 
\dint_{y}^{1}\left( t\left\vert f^{\prime \prime }(a)\right\vert ^{q}+\left(
1-t\right) \left\vert f^{\prime \prime }(b)\right\vert ^{q}\right) dt & = & 
\dfrac{\left( 1-y^{2}\right) \left\vert f^{\prime \prime }\left( a\right)
\right\vert ^{q}+\left( 1-y\right) ^{2}\left\vert f^{\prime \prime }\left(
b\right) \right\vert ^{q}}{2},%
\end{array}%
\end{equation*}%
which completes the proof.
\end{proof}

\begin{corollary}
\label{c2} Under the assumptions Theorem $\ref{thm2}$ with $y=\dfrac{1}{2},$ 
$c=0,$ $d=1,$ we have%
\begin{equation}
\begin{array}{l}
\left\vert \dfrac{1}{b-a}\dint_{a}^{b}f(x)dx-f\left( \dfrac{a+b}{2}\right)
\right\vert \\ 
\\ 
\text{\ \ \ \ \ \ \ }\leq \dfrac{\left( b-a\right) ^{2}}{16\left(
2p+1\right) ^{1/p}}\left\{ \left( \dfrac{\left\vert f^{\prime \prime }\left(
a\right) \right\vert ^{q}+3\left\vert f^{\prime \prime }\left( b\right)
\right\vert ^{q}}{4}\right) ^{1/q}+\left( \dfrac{3\left\vert f^{\prime
\prime }\left( a\right) \right\vert ^{q}+\left\vert f^{\prime \prime }\left(
b\right) \right\vert ^{q}}{4}\right) ^{1/q}\right\}%
\end{array}
\label{H10}
\end{equation}%
and if we take$\ y=c=d=\dfrac{1}{2}$ and $f^{\prime }(a)=f^{\prime }(b)$ in
Theorem $\ref{thm2},$ we have%
\begin{equation}
\begin{array}{l}
\left\vert \dfrac{1}{b-a}\dint_{a}^{b}f(x)dx-\dfrac{f\left( a\right)
+f\left( b\right) }{2}\right\vert \\ 
\\ 
\text{ \ \ \ \ \ \ \ \ \ }\leq \dfrac{\left( b-a\right) ^{2}}{16\left(
2p+1\right) ^{1/p}}\left\{ \left( \dfrac{\left\vert f^{\prime \prime }\left(
a\right) \right\vert ^{q}+3\left\vert f^{\prime \prime }\left( b\right)
\right\vert ^{q}}{4}\right) ^{1/q}+\left( \dfrac{3\left\vert f^{\prime
\prime }\left( a\right) \right\vert ^{q}+\left\vert f^{\prime \prime }\left(
b\right) \right\vert ^{q}}{4}\right) ^{1/q}\right\} .%
\end{array}
\label{H11}
\end{equation}
\end{corollary}

\begin{remark}
We note that the obtained midpoint inequalities (\ref{H10}) and (\ref{H11})
are better than the inequalities (\ref{H7}) and (\ref{H6}), respectively.
\end{remark}

\begin{theorem}
\label{thm3} Let $f:I\subset \mathbb{R}\rightarrow \mathbb{R}$ be twice
differentiable function on $I^{\circ }$ such that $f^{\prime \prime }\in
L_{1}[a,b]$ where $a,b\in I,$ $a<b$. If $\left\vert f^{\prime \prime
}\right\vert ^{q}$ is convex on $[a,b],$\ $q\geq 1$, then%
\begin{equation}
\begin{array}{l}
\left\vert \dfrac{\left( c-y\right) ^{2}-\left( d-y\right) ^{2}}{2}\left(
a-b\right) f^{\prime }(ya+(1-y)b)+\dfrac{\left( d-1\right) ^{2}f^{\prime
}(a)-c^{2}f^{\prime }(b)}{2}\left( a-b\right) \right.  \\ 
\\ 
\text{ \ \ \ \ \ \ \ \ \ }\left. +\left( c-d\right) f(ya+(1-y)b)+\left[
\left( d-1\right) f(a)-cf(b)\right] +\dfrac{1}{b-a}\dint_{a}^{b}f(x)dx\right%
\vert  \\ 
\\ 
\text{ \ \ \ \ \ \ \ \ \ }\leq \dfrac{\left( b-a\right) ^{2}}{6}\left\{
\left( c^{3}-\left( c-y\right) ^{3}\right) ^{1/p}\left( \dfrac{M\left\vert
f^{\prime \prime }\left( a\right) \right\vert ^{q}+N\left\vert f^{\prime
\prime }\left( b\right) \right\vert ^{q}}{4}\right) ^{1/q}\right.  \\ 
\\ 
\text{ \ \ \ \ \ \ \ \ \ }+\left. \left( (d-y)^{3}-\left( d-1\right)
^{3}\right) ^{1/p}\left( \dfrac{P\left\vert f^{\prime \prime }\left(
a\right) \right\vert ^{q}+Q\left\vert f^{\prime \prime }\left( b\right)
\right\vert ^{q}}{4}\right) ^{1/q}\right\} ,%
\end{array}
\label{1}
\end{equation}%
where%
\begin{equation*}
M=c^{4}-\left( c-y\right) ^{3}(c+3y),\ \ N=4c^{3}-c^{4}+\left( c-y\right)
^{3}(c+3y-4)
\end{equation*}%
\begin{equation*}
P=(d-y)^{3}(d+3y)-\left( d-1\right) ^{3}(d+3)\text{ and }Q=\left( d-y\right)
^{3}(4-d-3y)+(d-1)^{4}.
\end{equation*}
\end{theorem}

\begin{proof}
FromLemma $\ref{lm}$, by the definition $k\left( t\right) $ and using by
power mean inequality, it follows that%
\begin{equation*}
\begin{array}{l}
\left\vert \dfrac{\left( c-y\right) ^{2}-\left( d-y\right) ^{2}}{2}\left(
a-b\right) f^{\prime }(ya+(1-y)b)+\dfrac{\left( d-1\right) ^{2}f^{\prime
}(a)-c^{2}f^{\prime }(b)}{2}\left( a-b\right) \right. \\ 
\\ 
\text{ \ \ \ \ \ \ \ \ \ }\left. +\left( c-d\right) f(ya+(1-y)b)+\left[
\left( d-1\right) f(a)-cf(b)\right] +\dfrac{1}{b-a}\dint_{a}^{b}f(x)dx\right%
\vert \\ 
\\ 
\text{ \ \ \ \ \ \ \ \ \ }\leq \dfrac{\left( b-a\right) ^{2}}{2}\left\{
\dint_{0}^{y}\left( c-t\right) ^{2}\left\vert f^{\prime \prime
}(ta+(1-t)b)dt\right\vert dt+\dint_{y}^{1}\left( d-t\right) ^{2}\left\vert
f^{\prime \prime }(ta+(1-t)b)dt\right\vert dt\right\} \\ 
\\ 
\text{ \ \ \ \ \ \ \ \ \ }\leq \dfrac{\left( b-a\right) ^{2}}{2}\left\{
\left( \dint_{0}^{y}\left( c-t\right) ^{2}dt\right) ^{1/p}\left(
\dint_{0}^{y}\left( c-t\right) ^{2}\left\vert f^{\prime \prime
}(ta+(1-t)b)\right\vert ^{q}dt\right) ^{1/q}\right. \\ 
\\ 
\text{ \ \ \ \ \ \ \ \ \ }+\left. \left( \dint_{y}^{1}\left( d-t\right)
^{2}dt\right) ^{1/p}\left( \dint_{y}^{1}\left( d-t\right) ^{2}\left\vert
f^{\prime \prime }(ta+(1-t)b)\right\vert ^{q}dt\right) ^{1/q}\right\} . \\ 
\end{array}%
\end{equation*}%
Since $\left\vert f^{\prime \prime }\right\vert ^{q}$ is convex on $\left(
a,b\right) ,$ we known that for $t\in \left[ 0,1\right] ,$ 
\begin{equation*}
\left\vert f^{\prime \prime }(ta+(1-t)b)\right\vert ^{q}\leq t\left\vert
f^{\prime \prime }(a)\right\vert ^{q}+\left( 1-t\right) \left\vert f^{\prime
\prime }(b)\right\vert ^{q},
\end{equation*}%
thus, it follows that%
\begin{equation}
\begin{array}{l}
\left\vert \dfrac{\left( c-y\right) ^{2}-\left( d-y\right) ^{2}}{2}\left(
a-b\right) f^{\prime }(ya+(1-y)b)+\dfrac{\left( d-1\right) ^{2}f^{\prime
}(a)-c^{2}f^{\prime }(b)}{2}\left( a-b\right) \right. \\ 
\\ 
\text{ \ \ \ \ \ \ \ \ \ }\left. +\left( c-d\right) f(ya+(1-y)b)+\left[
\left( d-1\right) f(a)-cf(b)\right] +\dfrac{1}{b-a}\dint_{a}^{b}f(x)dx\right%
\vert \\ 
\\ 
\text{ \ \ \ \ \ \ \ \ \ }\leq \dfrac{\left( b-a\right) ^{2}}{2}\left\{
\left( \dfrac{c^{3}-\left( c-y\right) ^{3}}{3}\right) ^{1/p}\left(
\dint_{0}^{y}\left( c-t\right) ^{2}\left( t\left\vert f^{\prime \prime
}(a)\right\vert ^{q}+\left( 1-t\right) \left\vert f^{\prime \prime
}(b)\right\vert ^{q}\right) dt\right) ^{1/q}\right. \\ 
\\ 
\text{ \ \ \ \ \ \ \ \ \ }+\left. \left( \dfrac{\left( d-y\right)
^{3}-\left( d-1\right) ^{3}}{3}\right) ^{1/p}\left( \dint_{y}^{1}\left(
d-t\right) ^{2}\left( t\left\vert f^{\prime \prime }(a)\right\vert
^{q}+\left( 1-t\right) \left\vert f^{\prime \prime }(b)\right\vert
^{q}\right) dt\right) ^{1/q}\right\} .%
\end{array}
\label{2}
\end{equation}%
By simple computation,%
\begin{equation}
\begin{array}{l}
\dint_{0}^{y}\left( c-t\right) ^{2}tdt=\dfrac{c^{4}-(c-y)^{3}(c+3y)}{12} \\ 
\\ 
\dint_{0}^{y}\left( c-t\right) ^{2}(1-t)dt=\dfrac{%
4c^{3}-c^{4}+(c-y)^{3}(c+3y-4)}{12} \\ 
\\ 
\dint_{y}^{1}\left( d-t\right) ^{2}tdt=\dfrac{(d-y)^{3}(d+3y)-(d-1)^{3}(d+3)%
}{12} \\ 
\\ 
\dint_{y}^{1}\left( d-t\right) ^{2}(1-t)dt=\dfrac{(d-y)^{3}(4-d-3y)+(d-1)^{4}%
}{12}.%
\end{array}
\label{3}
\end{equation}%
Substituting (\ref{3}) into (\ref{2}) gives (\ref{1}).
\end{proof}

\begin{corollary}
\label{c3} Under the assumptions Theorem $\ref{thm3}$\ with $y=\dfrac{1}{2},$
$c=0,$ $d=1,$ we have%
\begin{equation}
\begin{array}{l}
\left\vert \dfrac{1}{b-a}\dint_{a}^{b}f(x)dx-f\left( \dfrac{a+b}{2}\right)
\right\vert \\ 
\\ 
\text{\ \ \ \ \ \ \ }\leq \dfrac{\left( b-a\right) ^{2}}{48}\left\{ \left( 
\dfrac{3\left\vert f^{\prime \prime }\left( a\right) \right\vert
^{q}+5\left\vert f^{\prime \prime }\left( b\right) \right\vert ^{q}}{8}%
\right) ^{1/q}+\left( \dfrac{5\left\vert f^{\prime \prime }\left( a\right)
\right\vert ^{q}+3\left\vert f^{\prime \prime }\left( b\right) \right\vert
^{q}}{8}\right) ^{1/q}\right\}%
\end{array}
\label{H12}
\end{equation}%
and if we take$\ y=c=d=\dfrac{1}{2}$ and $f^{\prime }(a)=f^{\prime }(b)$ in
Theorem $\ref{thm3},$ we have%
\begin{equation}
\begin{array}{l}
\left\vert \dfrac{1}{b-a}\dint_{a}^{b}f(x)dx-\dfrac{f\left( a\right)
+f\left( b\right) }{2}\right\vert \\ 
\\ 
\text{ \ \ \ \ \ \ \ \ \ }\leq \dfrac{\left( b-a\right) ^{2}}{48}\left\{
\left( \dfrac{\left\vert f^{\prime \prime }\left( a\right) \right\vert
^{q}+7\left\vert f^{\prime \prime }\left( b\right) \right\vert ^{q}}{8}%
\right) ^{1/q}+\left( \dfrac{7\left\vert f^{\prime \prime }\left( a\right)
\right\vert ^{q}+\left\vert f^{\prime \prime }\left( b\right) \right\vert
^{q}}{8}\right) ^{1/q}\right\} .%
\end{array}
\label{H13}
\end{equation}
\end{corollary}

\begin{remark}
If we take $q=1$ in (\ref{H12}) and (\ref{H13}), then we have the
inequalities (\ref{H8}) and (\ref{H9}), respectively.
\end{remark}

\section{Applications to Special Means}

We shall consider the following special means:

(a) The arithmetic mean: $A=A(a,b):=\dfrac{a+b}{2},$ \ $a,b>0,$

(b) The harmonic mean: 
\begin{equation*}
H=H\left( a,b\right) :=\dfrac{2ab}{a+b},\ a,b>0,
\end{equation*}

(c) The logarithmic mean: 
\begin{equation*}
L=L\left( a,b\right) :=\left\{ 
\begin{array}{ccc}
a & \text{if} & a=b \\ 
&  &  \\ 
\frac{b-a}{\ln b-\ln a} & \text{if} & a\neq b%
\end{array}%
\right. \text{, \ \ \ }a,b>0,
\end{equation*}

(d) The $p-$logarithmic mean

\begin{equation*}
L_{p}=L_{p}(a,b):=\left\{ 
\begin{array}{ccc}
\left[ \frac{b^{p+1}-a^{p+1}}{\left( p+1\right) \left( b-a\right) }\right] ^{%
\frac{1}{p}} & \text{if} & a\neq b \\ 
&  &  \\ 
a & \text{if} & a=b%
\end{array}%
\right. \text{, \ \ \ }p\in \mathbb{R\diagdown }\left\{ -1,0\right\} ;\;a,b>0%
\text{.}
\end{equation*}%
It is well known \ that $L_{p}$ is monotonic nondecreasing \ over $p\in 
\mathbb{R}$ with $L_{-1}:=L$ and $L_{0}:=I.$ In particular, we have the
following inequalities%
\begin{equation*}
H\leq L\leq A.
\end{equation*}%
Now, using the results of Section 2, some new inequalities is derived for
the above means.

\begin{proposition}
\label{p.1} Let $a,b\in R$, $0<a<b$ and $n\in \mathbb{N}$, $n>2.$ Then, we
have%
\begin{equation*}
\left\vert L_{n}^{n}\left( a,b\right) -A^{n}\left( a,b\right) \right\vert
\leq n(n-1)\frac{\left( b-a\right) ^{2}}{48}\left( a^{n-2}+b^{n-2}\right)
\end{equation*}%
and%
\begin{equation*}
\left\vert L_{n}^{n}\left( a,b\right) -A\left( a^{n},b^{n}\right)
\right\vert \leq n(n-1)\frac{\left( b-a\right) ^{2}}{48}\left(
a^{n-2}+b^{n-2}\right) .
\end{equation*}
\end{proposition}

\begin{proof}
The assertion follows from Corollary \ref{c1} applied to convex mapping $%
f\left( x\right) =x^{n},\;x\in \left[ a,b\right] $ and $n\in \mathbb{N}.$
\end{proof}

\begin{proposition}
Let $a,b\in R$, $0<a<b$ and $n\in \mathbb{N}$, $n>2.$ Then, we have%
\begin{multline*}
\left\vert L_{n}^{n}\left( a,b\right) -A^{n}\left( a,b\right) \right\vert \\
\leq n(n-1)\frac{\left( b-a\right) ^{2}}{16(2p+1)^{\frac{1}{p}}}\left\{
\left( \frac{a^{(n-2)q}+3b^{(n-2)q}}{4}\right) ^{\frac{1}{q}}+\left( \frac{%
3a^{(n-2)q}+b^{(n-2)q}}{4}\right) ^{\frac{1}{q}}\right\}
\end{multline*}%
and%
\begin{multline*}
\left\vert L_{n}^{n}\left( a,b\right) -A\left( a^{n},b^{n}\right) \right\vert
\\
\leq n(n-1)\frac{\left( b-a\right) ^{2}}{16(2p+1)^{\frac{1}{p}}}\left\{
\left( \frac{a^{(n-2)q}+3b^{(n-2)q}}{4}\right) ^{\frac{1}{q}}+\left( \frac{%
3a^{(n-2)q}+b^{(n-2)q}}{4}\right) ^{\frac{1}{q}}\right\} .
\end{multline*}
\end{proposition}

\begin{proof}
The assertion follows from Corollary \ref{c2} applied to convex mapping $%
f\left( x\right) =x^{n},\;x\in \left[ a,b\right] $ and $n\in \mathbb{N}.$
\end{proof}

\begin{proposition}
\label{p.2} Let $a,b\in R$, $0<a<b.$ Then, for all $q\geq 1$, we have%
\begin{multline*}
\left\vert L^{-1}\left( a,b\right) -A^{-1}\left( a,b\right) \right\vert  \\
\leq \frac{\left( b-a\right) ^{2}}{24}\left\{ \left( \frac{3a^{-3q}+5b^{-3q}%
}{8}\right) ^{\frac{1}{q}}+\left( \frac{5a^{-3q}+3b^{-3q}}{8}\right) ^{\frac{%
1}{q}}\right\} 
\end{multline*}%
and%
\begin{multline*}
\left\vert L^{-1}\left( a,b\right) -H^{-1}\left( a,b\right) \right\vert  \\
\leq \frac{\left( b-a\right) ^{2}}{24}\left\{ \left( \frac{a^{-3q}+7b^{-3q}}{%
8}\right) ^{\frac{1}{q}}+\left( \frac{7a^{-3q}+b^{-3q}}{8}\right) ^{\frac{1}{%
q}}\right\} .
\end{multline*}
\end{proposition}

\begin{proof}
The assertion follows from Corollary \ref{c3} applied to the convex mapping $%
f\left( x\right) =1/x,\;x\in \left[ a,b\right] .$
\end{proof}

\end{document}